\documentclass[10pt]{amsart}
\usepackage{amsmath,amsthm,amscd,amsfonts,amssymb,graphicx,paralist,color}
\usepackage{enumitem}
\usepackage[initials]{amsrefs}
\usepackage{xcolor}
\usepackage{hyperref}
\usepackage{rotating}
\hypersetup{colorlinks=false,linkbordercolor=red,linkcolor=green,pdfborderstyle={/S/U/W 1}}
\usepackage[skip=0.7\baselineskip]{caption}
\usepackage[utf8]{inputenc}
\usepackage[T1]{fontenc}
\usepackage{tikz}
\usetikzlibrary{trees}
\usepackage{enumerate}
\usepackage{geometry}
\theoremstyle{plain} 
\newtheorem{theorem}{Theorem}[section]
\newtheorem{lemma}[theorem]{Lemma} 
\newtheorem{corollary}[theorem]{Corollary} 
\newtheorem{proposition}[theorem]{Proposition}

\theoremstyle{definition} 
\newtheorem{remark}[theorem]{Remark}
\newtheorem{question}[theorem]{Question}

\theoremstyle{definition}

\newcommand{\charf}[1]{\mbox{\raise.48ex\hbox{$\chi$}$_{#1}$}}

\newcommand{\R}{\mathbb{R}}

\newcommand{\bN}{\mathbb{N}}

\DeclareMathOperator{\diam}{diam}

\newcommand{\BE}{\begin{equation}}
\newcommand{\EE}{\end{equation}}

\makeatletter
\@namedef{subjclassname@1991}{2020 Mathematics Subject Classification}
\makeatother

\title[Dimensions, Hölder regularity, and fractional differentiability]{
	Algebraic structures featuring graph dimensions, Hölder regularity, and fractional differentiability
}

\author[Esser]{C. Esser}
\address[C\'eline Esser]{\mbox{}\newline\indent Universit\'e de Li\`ege, \newline \indent D\'epartement de Math\'ematique, \newline \indent Quartier Polytech 1, All\'ee de la D\'ecouverte 12, B\^atiment B37, \newline \indent B-4000 Li\`ege, Belgique.}
\email{celine.esser@uliege.be}

\author[Maghsoudi]{S. Maghsoudi}
\address[S. Maghsoudi]{\mbox{}\newline \indent Department of Mathematics, \newline \indent  University of Zanjan, \newline \indent  Zanjan 45371-38791, Iran.}
\email{s\_maghsodi@znu.ac.ir}

\author[Rodr\'{i}guez]{D. L.~Rodr\'{i}guez--Vidanes}
\address[D. L.~Rodr\'{i}guez-Vidanes]{\mbox{}\newline\indent Grupo de Investigación de Análisis Matemático y Aplicaciones (AMA), \newline \indent Departamento de Matem\'{a}tica Aplicada a la Ingeniería Industrial, \newline \indent E.T.S.I.D.I.,  \newline \indent Ronda de Valencia 3, \newline \indent Universidad Politécnica de Madrid, \newline \indent	Madrid, 28012, Spain.}
\email{dl.rodriguez.vidanes@upm.es}

\author[Seoane]{J. B. Seoane--Sep\'ulveda}
\address[J. B. Seoane--Sep\'ulveda]{\mbox{}\newline\indent Instituto de Matem\'atica Interdisciplinar (IMI), \newline\indent Departamento de An\'{a}lisis Matem\'{a}tico y Matem\'atica Aplicada,\newline\indent Facultad de Ciencias Matem\'aticas,  \newline\indent Plaza de Ciencias 3, \newline\indent Universidad Complutense de Madrid, \newline\indent 28040 Madrid, Spain.}
\email{jseoane@ucm.es}

\begin{document}
	
	\subjclass{15A03; 46B87; 26A16; 26A27; 28A78}
	
	\keywords{lineability, algebrability, fractional differentiability, box dimension,  Hausdorff dimension, Hölder regularity,   continuous functions}
	\thanks{}

	\begin{abstract}
		We investigate the algebraic genericity of various families of continuous functions exhibiting extreme irregularity, focusing on fractal dimensions, Hölder regularity, and fractional differentiability. 
		Our first main result shows that for every $s \in (1,2]$, the set of continuous functions on $[0,1]$ whose graph has Hausdorff and box dimensions equal to $s$ is strongly $\mathfrak{c}$-algebrable, thereby tackling an open question from \cite{BMPS} by Bonilla et al., and complementing recent findings by Liu et. al \cite{LiZhSh} and Carmona et al. \cite{CFST22}. 
		We then extend the analysis to Hölder spaces: although the pointwise Hölder exponent of a generic function in $C^\alpha[0,1]$ is constant, we prove that the collection of functions realizing this behavior is $\mathfrak{c}$-lineable but cannot form an algebra. 
		Nevertheless, we construct strongly $\mathfrak{c}$-algebrable families of functions that exhibit Hölder exponent $\alpha$ outside a set of Hausdorff dimension zero.
		Finally, as a consequence of the relation between strongly monoH\"older functions and fractional differentiability, we analyze the strong $\mathfrak c$-algebrability of nowhere (Riemann-Liouville) fractional differentiable functions.
	\end{abstract}
		\maketitle

	\section{Introduction and preliminaries}
Weierstrass'  function is defined by
		$$
		W_{a,b}(x) = \sum_{k=0}^{\infty} a^k \cos (b^k 2 \pi x) ,
		$$
	with $0 < a < 1$ and  $ab \geq 1 $, and it is a classical example in real analysis of a function that is continuous everywhere but differentiable nowhere. 
	Besides, the graph of $W_{a,b}$ and some of its variations have been studied from a geometric point of view as fractal curves in the
	plane, see e.g. \cite{B12,BBR,H03,HL,H98,PU,MW86} and
	references therein. 
	For notation purposes, we will not distinguish between a function $f$ and its graph throughout this paper.
	While the box dimension of the graph of a function can be
	easily computed thanks to its H\"older regularity \cite{Falconer:97},
	the problem of the computation of the Hausdorff dimension is more
	complicated to tackle, and it is only  a few years ago that the
	conjecture related to the Hausdorff dimension of the graph of the
	Weierstrass' function has been solved \cite{keller,shen}. 	In particular,
		\begin{equation}\label{HausBox}
			\dim_{\mathcal{H}}(W_{a,b})  = \dim_{B}(W_{a,b})  = 2+\frac{\log(a)}{\log (b)},
		\end{equation}
	where $\dim_{\mathcal{H}}(\cdot)$ and $\dim_{B}(\cdot)$ denote the Hausdorff and box dimension, respectively.
	
	It may be considered as another shock that these pathological properties are not rare from several mathematical points of view.
	For instance, the residuality of the set of nowhere differentiable functions in the space of continuous function has been obtained in 1931 as a nice application of the Baire category theorem by Banach and Mazurkiewicz independently. 
	In 1994, Hunt \cite{H94} extended the latter result to the generic setting of prevalence, a concept introduced in order to generalize the notion of Lebesgue almost everywhere to infinite dimensional spaces \cite{Christensen74,HSY92}. 
	Let us remark that researches related to the residuality and prevalence of the set of	functions having a graph with a prescribed box or Hausdorff dimension
	have also been undertaken, see e.g. \cite{BH13,CN10,GJMNOP12,HLOP12}.

	Another notion that can show a pathological property is not algebraically rare is lineability. 
	The  theory of lineability has provided a number of concepts in order to quantify the existence of linear or algebraic structures inside a--not necessarily linear--set. 	 The term lineability was coined by V. I. Gurariy \cite{AGS} and is introduced in the Ph.D. Dissertation of the fourth author \cite{Seo}.
	It is a well-established line of research in mathematics now.	For a deeper understanding, we refer the interested reader to \cite{ar}. It is important to mention that, although lineability was first established by Gurariy, lineability results can be traced back to a paper of 1941 by Levin and Milman \cite{LeMi}. The algebraic structure of the set of nowhere differentiable functions in the space of continuous function  has also been deeply investigated using  this new  notion, see e.g. \cite{BQ,FGK99,Gurariy:66,jms}.
	Also, very recently several authors have analyzed the existence of large spaces with additional properties within the set of continuous nowhere differentiable functions \cite{FaPeRaRi,ArBaRaRi}.
	
	Working in this framework, Bonilla, Mu\~{n}oz-Fern\'{a}ndez, Prado-Bassas and the fourth author in \cite{BMPS} showed that for a given $s\in (1,2]$, the set of continuous functions whose graph has Hausdorff and box dimension equal to $s$ is $\mathfrak{c}$-lineable, where $\mathfrak{c}$ denotes the continuum, and the existing subspace can be chosen to be dense in $C[0,1]$, the Banach space of continuous functions on $[0,1]$ with the supreum norm (for definitions see Section \ref{prim}).
	They raised the following two questions (\cite[Questions~2.11 and~3.4]{BMPS}, respectively) of knowing whether besides  asking for vector subspaces, one could also study other structures, such as algebras. 
	
	\begin{question}\label{question1}
		Given $s \in (1, 2]$, is it possible to obtain the algebrability of the set of functions $f \in C[0,1]$ with $\dim_{\mathcal{H}}(f)=s$?
	\end{question}

	\begin{question}\label{question2}
		Given $s\in (1,2]$, is it possible to obtain maximal-dense-lineability (or algebrability) of the set of functions $f\in C[0,1]$ whose graph have box dimension $s$ everywhere in $[0, 1]$?
	\end{question}

	In \cite[Theorem~2.7]{LiZhSh}, Liu et. al answered Question~\ref{question2} in the affirmative.
	Observe that in terms of algebrability, Question~\ref{question1} is very general and does not require any specific conditions on the algebra.
	In this regard, Liu et. al in \cite[Theorem~2.7]{LiZhSh} showed that the set of functions $f \in C[0,1]$ with $\dim_{\mathcal{H}}(f)=s\in (1,2]$ is maximal-dense-lineable and also dense-algebrable.

	In this paper, after providing the necessary notions and some primary results in Section~\ref{prim},  we  obtain an affirmative answer to Question~\ref{question1} by showing that the the set of functions $f \in C[0,1]$ with $\dim_{\mathcal{H}}(f)=s\in (1,2]$ is strongly $\mathfrak{c}$-algebrable (this is an immediate consequence of Theorem \ref{thm:alggraph} in Section~\ref{sec:algebrability}).
	In Section~\ref{sec:lineab}, we analyze these types of problems in the context of Hölder spaces, which complement the findings of Liu et. al in \cite[Section~3]{LiZhSh}.
	In fact, we recall classical results concerning the H\"older pointwise regularity of a generic function in a given H\"older space, and then prove that the same holds with the notion of  lineability.
	We show that this generic behavior cannot hold on an algebra. 
	Nonetheless, we obtain a positive (but weaker) result if one allows the H\"older exponent to differ on a set of Hausdorff  dimension $0$. 
	The results of this section  are obtained using wavelets and Schauder bases, and in particular the characterization of  the H\"older regularity using these bases.
	As a consequence of the previous results and the relation established between notions of H\"older regularity and fractional differentiability, we study the strong $\mathfrak c$-algebrability of the family of nowhere fractional differentiable functions. Finally, we shall also complement and improve the results from \cite{CFST22} by addressing properties, such as fractal dimension, pointwise regularity, or fractional differentiability from the algebrability viewpoint.

\section{Preliminaries}\label{prim}
	
	In this section, we gather the definitions for the key notions of box and Hausdorff dimensions, various H\"older regularity conditions, fractional differentiability, lineability, and algebrability, besides some known results useful for the sequel.
	We shall use standard set-theoretical notation. As usual,  $\mathbb{N}$,  $\mathbb{Z}$, $\mathbb{Q}$ and $\mathbb{C}$ denote the sets of all  natural,  integer, rational and complex numbers, respectively. 
	Also, $\mathbb N_0$ denotes $\mathbb N\cup \{0\}$; and $\aleph_0$ and $\mathfrak c$ denote the cardinality of $\mathbb{N}$ and $\mathbb R$, respectively. 
	Given a function $f: A \to \mathbb R$, where $A$ is a nonempty subset of $\mathbb R$, and $B\subseteq A$, we denote the restriction of $f$ to $B$ by $f|_B$.

	\subsection{Hausdorff and box dimensions}
	We use    two different notions of dimension.  For more information, we refer the reader to \cite{Falconer:90,Falconer:97}. 
	
	\medskip
	
	Let $E \subseteq \R^d$ be bounded. For any $\delta >0$, we denote by $N_\delta (E)$ the minimal number of cubes of size $\delta$ needed to cover $E$. 
	If it exists, the \emph{box dimension} of $E$ is given by the limit
		$$
		\dim_B(E) := \lim_{ \delta \to 0^{+}} \frac{\log N_\delta (E)}{- \log \delta}.
		$$
	If the limit does not exist, one considers instead the upper and lower
	box dimensions, denoted respectively by $\overline{\dim}_B(E)$ and $\underline{\dim}_B(E)$, and defined using upper and lower limits.
	 
	 \medskip
	
	The notion of dimension which is mainly used in multifractal analysis is the Hausdorff dimension. 
	 
	Let $E \subseteq \R^d$ and $\delta >0$. If $E \subseteq \bigcup_{i \in  \mathbb N} E_i$ with $0 \leq \diam (E_i) \leq \delta$ for every $i \in \mathbb N$, we say that $(E_i)_{i \in \mathbb N}$ is a \textit{countable $\delta$-covering} of $E$.  For every $r  \geq 0$ and $\delta>0$, one sets
		\[
		\mathcal{H}^{r}_{\delta} (E) :=  \inf \left\{ \sum_{i=1}^\infty \diam(E_i)^{r} : (E_i)_{i \in \mathbb N} \mbox{ countable $\delta$-covering of } E \right\},
		\]
	where $\diam(\cdot)$ denotes the diameter of a set.
	Since $\mathcal{H}^{r}_{\delta}(E)$ is a decreasing function with respect to $\delta$, one can define the \emph{$r$-dimensional Hausdorff outer measure} of $E$ by setting
		\[
		\mathcal{H}^{r} (E) := \sup_{\delta > 0} \mathcal{H}^{r}_{\delta} (E) = \lim_{\delta \rightarrow 0^+} \mathcal{H}^{r}_{\delta} (E).
		\]
	The \emph{Hausdorff dimension} of $E$ is the unique value $\dim_{\mathcal{H}}(E)$ such that
	\[
	\mathcal{H}^{r} (E) = \left\{
	\begin{array}{ll}
		+ \infty & \mbox{ if } r < \dim_{\mathcal{H}}(E), \\
		0 & \mbox{ if } r > \dim_{\mathcal{H}}(E).
	\end{array}
	\right.
	\]
	Moreover, we use the convention that $\dim_{\mathcal{H}}(\varnothing) = - \infty$. 
	For any non-empty bounded set $E \subseteq \R^d$, one  has
	\BE \label{eq:reldim}
	0 \leq \dim_{\mathcal{H}} (E) \leq \underline{\dim}_B(E) \leq \overline{\dim}_B(E) \leq d.
	\EE
	Note that any function $f$ is contained in $\R^2$ and, therefore, $d=2$ in the latter chain of inequalities.
 
 	\medskip
	
	Let us present the link between the Hölder regularity of a function and the
	box dimension of its graph, see e.g. \cite{Falconer:97}. 
	First, let us recall that given $\alpha \geq 0$ and $a<b$,  a function $f: [a,b] \to \R$ belongs to the
	\emph{H\"older space} $C^{\alpha}[a,b]$ if there exists a constant
	$C>0$ such that
	$$
	|f(x) - f(y)| \leq C |x-y|^{\alpha}
	$$
	for every $x,y \in [a,b]$. 
	In this case, we say that $f$ is \textit{Hölder continuous of exponent $\alpha$}, or simply that $f$ has \textit{Hölder regularity $\alpha$}.
	Moreover, a function $f$ is \textit{Hölder regular} if there is an $\alpha \geq 0$ such that $f$ has Hölder regularity $\alpha$.
	Classically, the H\"older space $C^{\alpha}[a,b]$
	is endowed by the norm
	$$
	\|f\|_{C^{\alpha}[a,b]} := \|f\|_{L^{\infty}[a,b]} + \sup_{x,y \in
		[a,b]} \frac{|f(x)-f(y)|}{|x-y|^{\alpha}}.  
	$$ 
	If $\alpha =1$, we rather say that $f$ is Lipschitz.
		
	Let us recall a known result about the box dimension of functions in  H\"older spaces.
	 
	\begin{proposition}\cite{Falconer:97}.\label{prop_dimgraph}
		Let $a<b$ be two real numbers and let $f: [a,b] \to \R$ be a continuous function.
		\begin{itemize}
			\item[\rm{(i)}] If there is an $\alpha \in [0,1]$ such that $f \in
			C^{\alpha}[a,b]$, then $ \overline{\dim}_B (f) \leq  2- \alpha.$
			\item[\rm{(ii)}] If there are $\alpha \in [0,1]$, $\delta_0>0$ and $c>0$ such that for every $x \in [a,b]$ and every $\delta \in (0,\delta_0]$, there is a $y \in [a,b]$ with $|x-y| \leq \delta$ and
			$$
			\big| f(x) - f(y)\big| \geq c \, \delta^{\alpha}, 
			$$
		then $\underline{\dim}_B (f)\geq 2-\alpha$.
		\end{itemize}
	\end{proposition}
	
	Any function satisfying conditions (i) and (ii) of Proposition
	\ref{prop_dimgraph} is called \emph{strongly monoH\"older} on $[a,b]$
	with exponent $\alpha$.
	Note that the box dimension of any strongly monoH\"older function of exponent $\alpha$ exists and it is equal to $2-\alpha$.  
	However, it is not true in general that the Hausdorff dimension of its graph is $2-\alpha$ (see, e.g., \cite{MR2898734,BBR}).
	
	Of course, the graph of any Lipschitz function has Hausdorff and box dimensions equal to $1$. In fact, the addition of Lipschitz 
	functions does not modify the Hausdorff dimension of  graphs, as recalled in the next result.
	
	\begin{proposition}\cite{MW86}\label{prop:dimLip}
		Let $g$ be a Lipschitz %continuous 
		function on $[a,b]$. 
		Then, for every function $f: [a,b]\to \R$, we have
		$$
		\dim_{\mathcal{H}}(f+g) = \dim_{\mathcal{H}}(f). 
		$$
	\end{proposition}

	\medskip

	 The Hölder regularity of nowhere differentiable functions can also be studied locally by the notion of H\"older pointwise regularity.
	 While the regularity of the Weierstrass' function is the same at every point \cite{Hardy:16}, there exist functions whose regularity varies significantly from point to point, see among others \cite{Arneodo:98,Balanca:14,Jaffard:96,Jaffard:04}. 
	 For completeness, let us recall the notion of Hölder pointwise regularity. 
		 
	Let $x_0 \in \R$ and $\alpha \geq 0$.  
	A locally bounded function $f$ belongs to $C^\alpha(x_0)$  if there exist a constant $C>0$ depending on $x_0$ and a polynomial $P_{x_0}$ depending on $x_0$ of degree less than $[\alpha]$, where $[\cdot]$ denotes the integer part, such that 
		\begin{equation}\label{Hol}
			|f(x) - P_{x_0} (x) | \leq C |x-x_0|^\alpha
		\end{equation}
	for every $x$ in a neighborhood of $x_0$.  
	The {\sl H\"older exponent} $h_f(x_0)$ of $f$ at $x_0$ is defined as its maximal Hölder regularity at $x_0$, i.e.,
		\[ 
		h_f(x_0) := \sup \big\{ \alpha \geq 0 : f \in C^\alpha(x_0) \big\},
		\]
	possibly equal to $+ \infty$. 
	The notion of $h_f(x_0)$ allows to quantify the local smoothness of $f$ at $x_0$.
	Observe that when $h_f(x_0) \leq 1$ (which is the case if $f$ is not differentiable at $x_0$), the H\"older exponent  is simply given by the formula
		\[
		h_f(x_0) = \liminf_{x \to x_0} \frac{\log |f(x) - f(x_0)|}{\log |x-x_0|} .
		\]
	While the pointwise regularity is fully encapsulated by the H\"older exponent, the global H\"older regularity can be characterized by H\"older spaces.
	Also, note that $f \in C^{\alpha}[a,b]$, with $a<b$ and $\alpha \geq 0$, if \eqref{Hol} holds for every $x_0 \in [a,b]$ and with the constant $C$ being uniform. 
	
	\medskip
		
	Now, let us recall some basic notions from fractional calculus found in \cite{kol}.
	
	\subsection{Riemann-Liouville fractional derivative}\label{fracder}
	
	Let $f:I \rightarrow \mathbb{R}$, with $I\subseteq \mathbb R$ an open interval, and $0<q<1$ be a number. The \textit{Riemann-Liouville fractional derivative of order $q$} of $f$ at $a\in I$ is defined by
		\[
		\frac{d^q f(x)}{[d(x-a)]^q} := \frac{1}{\Gamma (1-q)} \cdot \frac{d}{dx}\int_a^x \frac{f(y)}{(x-y)^q}dy,
		\]
	where $\Gamma$ denotes the Gamma function.
	Following \cite{kol}, we define the \textit{local fractional derivative of order $q$} of the function $f$ at $a$  by
			\[
			\mathbb{D}^qf(a) := \lim_{x\to a}\frac{d^q(f(x)-f(a))}{[d(x-a)]^q}
			\]
	provided the limit exists and is finite (in which case we say that $f$ is \textit{locally fractional differentiable of order $q$} at $a$).
	We say that that a function $f$ from $I$ to $\R$ is \textit{fractional differentiable of order $q$} if it is locally fractional differentiable of order $q$ at every $a\in I$.
	Also, we will say that $f$ is \textit{nowhere fractional differentiable of order $q$} if it is not locally fractional differentiable of order $q$ at any $a\in I$.
	
	In \cite{kol}, Kolwankar and Gangal established a relation between condition (ii) of Proposition~\ref{prop_dimgraph} and the fractional nowhere differentiability of functions.
	To be more precise, they show that if a function $f: [a,b] \to \R$ satisfies condition (ii) of Proposition
	\ref{prop_dimgraph} for some $\alpha \in (0,1)$, then $f$ is nowhere fractional differentiable on $(a,b)$ for any order $q\in (\alpha,1)$.
	
	Regarding the fractional differentiability of order $q<\alpha$, Kolwankar and Gangal also proved in \cite{kol} that if a function $f\in C^{\alpha}[a,b]$ for some $\alpha \in (0,1)$, then $f$ is fractional differentiable on $(a,b)$ for any order $q\in (0,\alpha)$.
	
	Thus, as a consequence of Proposition~\ref{prop_dimgraph}, any strongly monoH\"older function $f$ on $[a,b]$ with exponent $\alpha$ is fractional differentiable on $(a,b)$ for any order in $(0,\alpha)$ and nowhere fractional differentiable on $(a,b)$ for any order in $(\alpha,1)$.
	
\subsection{Lineability and algebrability}
	
	To finish this section, we present some definitions and techniques from lineability theory that enable us to formalize our results correctly.   
	First, we have the following notion. 
	
	Let $X$ be a vector space, $M$ a subset of $X$, and $\kappa$ a cardinal
	number, we say that $M$ is {\em $\kappa$-lineable} if $M \cup \{0\}$ contains a vector subspace $V$ of $X$ of dimension $\kappa$.  
	The set $M$ is simply {\em lineable} if the existing subspace is infinite dimensional \cite{AGS,GQ,Seo}. 
	When $X$ is a topological vector space and $V$ can be chosen to be dense in $X$, we  say that $M$ is {\em $\kappa$-dense-lineable} (or, simply, dense-lineable if $\kappa$ is infinite).
		
	We also need the notion of algebrability introduced in  \cite{APS, AS}. If  $\mathcal{A}$ is an algebra
	and  $\mathcal{B}$ is a subset of $\mathcal{A}$, we say that the set
	$\mathcal{B}$ is {\em $\kappa$-algebrable}  if $\mathcal{B} \cup \{
	0\}$ contains a $\kappa$-generated subalgebra $\mathcal{C}$ of
	$\mathcal{A}$.
	If $\kappa$ can be chosen to be infinite, we simply say that $\mathcal B$ is \textit{algebrable}.
	Of course, any algebrable set is, automatically, lineable as well.
	In general, the converse is false (see, e.g., \cite{ar} for a collection of examples illustrating this statement).
		
	Let us now recall one of the strongest notions in algebrability introduced in \cite{BG1}.
	Given a commutative algebra $\mathcal{A}$ and a cardinal number $\kappa$, a subset $\mathcal{B} \subseteq \mathcal{A}$ is {\em strongly \(\kappa\)-algebrable}  if there exists a \(\kappa\)-generated free algebra $\mathcal C$ contained in $\mathcal{B} \cup \{ 0\}$. We recall that a subset $X$ of a commutative algebra generates a {\em free subalgebra} if for each polynomial $P$ without a constant term and any $x_1, \dots, x_n \in X$, we have $P(x_1, \dots, x_n) =0$ if and only if $P=0$ (that is, the set of all elements of the form $x_1^{k_1} \dots x_n^{k_n}$ where $x_1, \dots, x_n \in X$ and where $k_1, \dots, k_n \in \bN_0$ are not all equal to $0$, is linearly independent).
	
	Finally, let us recall a technique developed in \cite{BBF}, the so-called {\em exponential-like function method}, that  allows to obtain the strong algebrability of some sets of functions defined on $[0,1]$, see \cite{BBF,GGMS}.   Exponential-like functions are particularly useful for constructing free algebras due to their distinct growth properties. 	We say that a function $g : \R \rightarrow \R$ is {\em exponential-like (of range $m$)} whenever $g$ is given by
	\[
	g(x) = \sum_{i=1}^m a_i e^{\beta_ix}
	\]
	for some distinct non-zero real numbers $\beta_1, \dots, \beta_m$ and
	some non-zero real numbers $a_1, \dots, a_m$.  These functions are injective modulo finite exceptions (Lemma \ref{lemma_explike}), ensuring algebraic independence when composed with suitable base functions like Weierstrass' monster functions. 

	In \cite{BBF}, the authors proved the following very useful property of exponential-like functions. 
	
		\begin{lemma} \cite{BBF} \label{lemma_explike}
			For every  $m \in \bN$, every exponential-like function $g$ of range $m$ and every $c \in \R$, the level set $g^{-1}(\{c\})$ has at most $m$ elements and 				there exists a finite decomposition of $\R$ into intervals such that $g$ is strictly monotone in each of them.
		\end{lemma}. 
	
	Finally, the following general result for strong algebrability was obtained in \cite{BBF}
	
		\begin{proposition} \cite{BBF} \label{prop_explike}
			Let $\mathcal{F}$ be a family of functions from $[0,1]$ to $\R$ and assume that there is an $F \in \mathcal{F}$ such that $g \circ F \in \mathcal{F}\setminus \{0\}$ for any exponential-like function $g$.
			Then, $\mathcal{F}$ is strongly $\mathfrak{c}$-algebrable.
			In particular, if $\mathcal{H}$ is a Hamel basis of $\R$, then $\{  \exp \circ (rF) : r \in \mathcal{H} \}$ is a system of generators of a free algebra contained in $\mathcal{F} \cup \{0\}$.
		\end{proposition}

	\section{Box and Hausdorff dimensions of continuous functions}\label{sec:algebrability}

	In this section, we obtain two results of algebrability related with the dimension of the graph of functions in $C^{\alpha}[0,1]$. 
	Our aim in this section is to answer affirmatively Question~\ref{question1} and complement \cite{LiZhSh}. 
	More precisely, for every $s \in (1,2]$, the aim is  to construct a subalgebra of $C[0,1]$ generated by $\mathfrak{c}$-many algebraically independent functions whose graphs have box and Hausdorff dimensions
	exactly equal to  $s$. 
	The strategy we adopt is the following: we will work in the H\"older space $C^{\alpha}[0,1]$, where $\alpha = 2-s$, which is a subalgebra of $C[0,1]$ and has the property that the graph of every function of $C^{\alpha}[0,1]$ has an upper box dimension bounded by $s$ by Proposition \ref{prop_dimgraph}.
	We will then only have to control the dimension from below. 
	
	Note that a result of prevalence concerning the dimensions of graphs of functions belonging $C^{\alpha}[0,1]$ has already been obtained.
	In \cite{CN10}, it is indeed proved that in the sense of prevalence,  the typical  behavior of the graph of functions in $C^{\alpha}[0,1]$ is as irregular as allowed by the global H\"oder regularity, i.e., the Hausdorff dimension of the graph is $2-\alpha$. Also,  it is proved in \cite[Theorem 11.6]{Falconer:86} that
	the graph of a
	Baire generic function in $C^{0}[0,1]$ has a Hausdorff dimension
	equal to $1$. 
		
	For completeness, we begin by proving the following well-known lemma on the Hausdorff dimension of Weierstrass' function.
	
	\begin{lemma}\label{hausdorffweierstrass}
		Let $0<a<1$, $b$ be an integer such that $ab \geq 1$ and $\displaystyle \frac{\log(a)}{\log (b)} = - \alpha $, and $I\subseteq [0,1]$ a nondegenerate interval.
		Then,
			$$
			\dim_{\mathcal{H}}(W_{a,b}\vert_I) = 2 - \alpha.
			$$
	\end{lemma}

	\begin{proof}
		Take $k_0\in \mathbb N$ large enough so that the length of  $I$ is larger than $b^{-k_0}$.
		Set 
			$$
			W^{k_0}_{a,b}(x) := \sum_{k=k_0}^{\infty} a^k \cos (b^k 2 \pi x) .
			$$
		Now define
			$$
			g := \sum_{k=0}^{k_0-1} a^k \cos(b^k 2\pi x).
			$$
		Note that $W_{a,b} = W_{a,b}^{k_0}+ g$ and $g$ is Lipschitz.
		Thus, by Proposition~\ref{prop:dimLip}, we have
			$$
			\dim_{\mathcal{H}}(W_{a,b}) = \dim_{\mathcal{H}}(W_{a,b}^{k_0}+ g) = \dim_{\mathcal{H}}(W_{a,b}^{k_0}).
			$$
		Moreover, since the restriction of a Lipschitz function to any interval is Lipschitz, we have again from Proposition~\ref{prop:dimLip} that
			$$
			\dim_{\mathcal{H}}(W_{a,b}\vert_I) = \dim_{\mathcal{H}}(W_{a,b}^{k_0}\vert_I + g\vert_I) = \dim_{\mathcal{H}}(W_{a,b}^{k_0}\vert_I).
			$$
		Now, as $W^{k_0}_{a,b}$ is $b^{-k_0}$-periodic, the restriction $W^{k_0}_{a,b}\vert_I$ satisfies
			$$
			\dim_{\mathcal{H}}(W^{k_0}_{a,b}\vert_I) =    \dim_{\mathcal{H}}(W^{k_0}_{a,b}).
			$$
		Consequently, by \eqref{HausBox}, it yields
			$$
			\dim_{\mathcal{H}}(W_{a,b}\vert_I) = 2 - \alpha. 
			$$
	\end{proof}
		
	\begin{theorem}\label{thm:alggraph}
		Let $\alpha \in  (0,1)$. 
		The set of functions $f \in
		C^{\alpha}[0,1]$ with 
			$$
			\dim_{\mathcal{H}}(f) = \dim_{B} (f)   = 2-\alpha
			$$
		is strongly $\mathfrak{c}$-algebrable.
	\end{theorem}

	\begin{proof}
		Let $g$ be an exponential-like function. Using Proposition~\ref{prop_explike}, it suffices to prove that 
			$$
			\dim_{\mathcal{H}}(g \circ W_{a,b})  =
			\dim_{B} (g \circ W_{a,b}) =  2- \alpha.
			$$
		Note that $g \circ W_{a,b} \in C^\alpha [0,1]$ since 
			$$
			\big|g \circ W_{a,b}(x) - g \circ W_{a,b} (y)\big|  \leq \|g'\|_{L^{\infty}[0,1]}
			\|W\|_{C^{\alpha}[0,1]} |x-y|^{\alpha}
			$$
		for all $x,y \in [a,b]$. 
		Hence, by Proposition~\ref{prop_dimgraph}(i), it implies that
			$$
			\overline{\dim}_B (g \circ W_{a,b}) \leq  2- \alpha.
			$$
		From \eqref{eq:reldim}, it is enough to prove that 
			$$
			\dim_{\mathcal{H}}(g \circ W_{a,b})  \geq  2- \alpha.
			$$
		Lemma \ref{lemma_explike} and the intermediate value theorem applied to $W_{a,b}$ guarantee the existence of an interval $I \subseteq [0,1]$ such that $g$ is strictly monotone on $W_{a,b}(I)$. 
		In particular, it is invertible on
		$W_{a,b}(I)$. Moreover,  since $g$ is continuously differentiable on
		$W_{a,b}(I)$ and $g'$ does not vanish on $W_{a,b}(I)$, its
		inverse is also continuously differentiable. Hence, the function
		$$
		(x,y) \in I \times W_{a,b}(I) \mapsto \big(x,g(y)\big) \in I \times g\big(W_{a,b}(I)\big)
		$$
		is bi-Lipschitz. 
		Since the Hausdorff dimension is
		invariant under bi-Lispchitz maps (see, e.g., \cite{Falconer:97}), we obtain from Lemma~\ref{hausdorffweierstrass} that
			$$
			\dim_{\mathcal{H}}(g \circ W_{a,b})  \geq
			\dim_{\mathcal{H}}(g \circ W_{a,b}\vert_{I}) =
			\dim_{\mathcal{H}}(W_{a,b}\vert_{I}) = 2- \alpha, 
			$$ 
		which concludes the proof. 
	\end{proof}
		
	\begin{remark}
		Following \cite{BMPS}, we say that a function has box dimension
		(resp. Hausdorff dimension) $\alpha$ \emph{everywhere on $[0,1]$} if
			$$
			\dim_{B} (f|_{[a,b]}) = \alpha  \quad (\text{resp. }\dim_{\mathcal{H}} (f|_{[a,b]}) = \alpha \,)
			$$
		for all $[a,b]\subseteq [0,1]$.
		For instance, Weierstrass' functions satisfy this homogeneity property. 
		A direct adaptation of the previous proof shows that  the set of functions of $C^{\alpha}[0,1]$ which have box and Hausdorff dimensions $\alpha$ everywhere on $[0,1]$ is strongly $\mathfrak{c}$-algebrable.
	\end{remark}
		
	As a direct consequence, we obtain the following result, answering Question~\ref{question1} raised in \cite{BMPS} and complementing \cite[Theorem~2.7]{LiZhSh}. 
	
	\begin{corollary}
		Let $s \in  (1,2]$. The set of functions $f \in C[0,1]$
		with
			$$
			\dim_{\mathcal{H}}(f) = \dim_{B} (f) = s
			$$
		is strongly $\mathfrak{c}$-algebrable.
	\end{corollary}
	
	\begin{proof}
		If $s<2$, it suffices to consider the algebra given by Theorem
		\ref{thm:alggraph} for $\alpha = 2-s$.  If $s=2$, let  $F$ be a continuous function whose graph has
		Hausdorff dimension $2$ everywhere in $[0,1]$ (see, e.g., \cite{W95}). Then, by Proposition \ref{prop_explike}, it suffices to prove that 
		$$
		\dim_{\mathcal{H}}(g \circ F) \geq 2.
		$$
		for any exponential-like function $g$. 
		We proceed analogously as in the proof of Theorem~\ref{thm:alggraph} to obtain that $\dim_{\mathcal{H}}(g
		\circ F) \ge 2$, the other inequality being obvious. 
	\end{proof}

		\section{H\"older pointwise regularity and fractional differentiability}\label{sec:lineab}

	Similarly to what was done in the case of nowhere differentiable functions, the size of the set of functions having the same Hölder exponent $\alpha$ belonging to the H\"older space of exponent $\alpha$ has been investigated using the concepts of Baire residuality and prevalence \cite{H94,Jaffard00}. 
	The authors obtained that a generic function is as irregular as allowed by the regularity of the space at every point.
	All these results have been obtained via wavelet decomposition  of functions and thanks to a characterization of the regularity based on the wavelet coefficients. 
	We summarize below the results obtained in \cite{H94,Jaffard00}.

	\begin{proposition}\cite{H94,Jaffard00}\label{T2} 
		The set of functions $f \in C^{\alpha}[0,1]$ satisfying  $h_f (x) = \alpha$ for every $x \in [0,1]$ is residual and prevalent in $C^{\alpha}[0,1]$.
	\end{proposition}

	It is worth mentioning at this point that in the complex case, Bernal-González, Bonilla, López-Salazar and the fourth author in \cite{BeBoLoSe} analyzed the existence of large spaces of nowhere H\"olderian functions with additional topological properties.
	
	\medskip
		
	The question we will analyze in this section is the following: 
		\begin{itemize}
			\item Is it possible to obtain lineability and algebrability  of the set of functions $f\in C^{\alpha}[0,1]$ satisfying  $h_f (x) = \alpha$ for every $x \in [0,1]$?
		\end{itemize}
		
	We will obtain positive and negative results, and see in particular that one cannot hope to construct an algebra of such functions. 
	First, we start by recalling some useful definitions and  results.

\medskip

We now focus on Schauder basis, which provides a convenient tool to study continuous functions on $[0,1]$.  Let $\Lambda$ denote the \emph{hat function} defined by
\[
\Lambda(x) := 
\begin{cases}
\min(x,1-x) & \text{if } x \in [0,1],\\
0 & \text{otherwise.}
\end{cases}
\]
Then every continuous function $f:[0,1]\to \R$ can be written as
\[
f(x) = f(0) + (f(1)-f(0))x + \sum_{j \in \bN_0} \sum_{k=0}^{2^j-1} c_{j,k} \Lambda(2^j x - k),
\]
where the convergence is uniform on $[0,1]$, and the coefficients are given by
\[
	c_{j,k} =  2  f\left( \frac{2k+1}{2^{j+1}}\right) - f \left(\frac{k}{2^j} \right) - f\left(\frac{k+1}{2^j} \right).
\]

\medskip

Using the coefficients in the Schauder basis, one can characterize the regularity of $f$ both globally and locally, see \cite{DLVM98}:

\begin{itemize}
\item {Global H\"older regularity:} $f \in C^\alpha[0,1]$ if and only if
\[
\sup_{j \in \bN_0} \sup_{k \in \{0, \dots, 2^j-1\}} |c_{j,k}| \, 2^{\alpha j} < \infty.
\]

\item {Pointwise H\"older regularity:} If $f \in C^\alpha(x_0)$, then there exists a constant $C>0$ such that
\begin{equation}\label{schaudercaract}
|c_{j,k}| \le C \, 2^{-\alpha j} \big(1 + |2^j x_0 - k|\big)^\alpha,
\end{equation}
for all $j \in \mathbb N_0$ and $k \in \{0, \dots, 2^j-1\}$.
Conversely, if \eqref{schaudercaract} holds and if there is $\varepsilon>0$ such that $f \in C^{\varepsilon}[0,1]$, then $f \in  C^\beta(x_0)$ for every $\beta < \alpha$.
\end{itemize}

		\medskip

	Let us start provide the lineability of the set considered in Proposition \ref{T2}.
	
	\begin{proposition} \label{prop:lin}
		The set of functions $f \in C^{\alpha}[0,1]$ such that $h_{f}(x) = \alpha$ for every
		$x \in [0,1]$ is $\mathfrak{c}$-lineable.
	\end{proposition}

\begin{proof}
Let $f \in C^{\alpha}[0,1]$ be such that $h_{f}(x) = \alpha$ for every $x \in [0,1]$, and denote by $c_{j,k}$ its Schauder coefficients. 

For each $a>0$, define
\[
f_a(x) := \sum_{j\in \bN_0} \sum_{k=0}^{2^j-1} \frac{1}{(j+1)^a} c_{j,k} \Lambda(2^j x - k).
\]
By the Schauder coefficient characterization of H\"older regularity, it yields that $f_a \in C^\alpha[0,1]$ with $h_{f_a}(x) = \alpha$ for all $x$, since its Schauder coefficients are exactly $\frac{1}{(j+1)^a} c_{j,k}$.

Now, consider any nontrivial finite linear combination of the functions $f_a$. 
Let $n \ge 1$, with parameters $a_n > \dots > a_1 > 0$ and nonzero scalars $\beta_1, \dots, \beta_n$, and define
\[
g := \sum_{i=1}^n \beta_i f_{a_i} = \sum_{j\in \bN_0} \sum_{k=0}^{2^j-1} \sum_{i=1}^n \frac{\beta_i}{(j+1)^{a_i}} c_{j,k} \Lambda(2^j x - k).
\]
The Schauder coefficients of $g$ are then
\[
d_{j,k} := \sum_{i=1}^n \frac{\beta_i}{(j+1)^{a_i}} c_{j,k}.
\]
For $j$ large enough, these coefficients satisfy
\[
\frac{|\beta_1|}{2 (j+1)^{a_1}} |c_{j,k}| \le |d_{j,k}| \le \frac{2|\beta_1|}{(j+1)^{a_1}} |c_{j,k}|,
\]
for all $k \in \{0,\dots,2^j-1\}$.
Thus, by the characterization of the H\"older exponent via Schauder coefficients, $g$ is not identically zero and satisfies $h_g(x) = \alpha$ for every $x \in [0,1]$.

This shows that the set of such functions is $\mathfrak{c}$-lineable.
\end{proof}

	Let us now turn to questions of algebrability. 
	The following results show that in any algebra of locally bounded functions, there are always functions with arbitrarily large pointwise regularity.  
	As a consequence, the set satisfying the hypotheses of Proposition~\ref{prop:lin} is not even $1$-algebrable. 
	
	\begin{proposition}
		Let $f : [0,1]\to \R$ be a locally bounded function. For every $x_0 \in [0,1]$ and $\alpha >0$, there exists a function in the algebra generated by $f$ with a H\"older exponent at $x_0$ larger than $\alpha$.
	\end{proposition}
	
	\begin{proof}
			Let $x_0$ be an arbitrary point in $[0,1]$. Then, there exists a function $g$ in the algebra generated by $f$ such that $g(x_0) =0$.
		Indeed, if it is not the case for $f$, it suffices to consider the function $g(x) = f(x)^2 -f(x_0)f(x)$. For $n\in \bN$ large enough, the H\"older exponent of $g^n$ at $x_0$ will be  arbitrarily  large.
		Hence, for $n\in \bN$ large enough, we have $h_{g^n}(x_0) \geq \alpha$. 
	\end{proof}

	\begin{corollary}
		The set of functions $f \in C^{\alpha}[0,1]$ such that $h_{f}(x) = \alpha$ for every
		$x \in [0,1]$ is not $1$-algebrable.
	\end{corollary}
	
	\begin{remark}
		One can conclude from the latter proof that the pointwise regularity will considerably be modified if the level sets $f^{-1}(\{y\})$ of $f$ are ``large''.
		 In \cite[Section~3]{LiZhSh}, Liu et al. analyze level sets in the context of Hölder regularity, but their focus differs from ours. 
		 While they study generic properties of level sets for prevalent functions, we explicitly construct functions (Lemma \ref{lemma_constr}) with level sets of Hausdorff dimension zero, ensuring minimal overlap and preserving regularity under algebraic operations. 
		 This complements their results by providing concrete examples with controlled level sets.
		 But one can hope to get a modification of the H\"older pointwise regularity only on a small subset of $[0,1]$ by controlling the level sets of $f$.
	\end{remark}

	The following lemma is inspired by some constructions presented in \cite{KK09}.
	\begin{lemma} \label{lemma_constr}
		Assume that $\alpha \in  (0,1)$. There exists a function $F \in C^{\alpha}[0,1]$ strongly monoH\"older with exponent  $\alpha$ on $[0,1]$ and whose level sets satisfy $\dim_{\mathcal{H}}F^{-1}(\{y\})=0$ for every $y \in \R$. 
	\end{lemma}

	\begin{proof}
		We will construct $F$ via its coefficients in the Schauder basis. First, let us consider a sequence $(r_n)_{n \in \mathbb N}$ of strictly positive real numbers which decreases to $0$. By recurrence, we construct a strictly increasing sequence $(j_n)_{n \in \mathbb N_0}$ of numbers in $\mathbb N_0$ as follows: we fix $j_0 = 0$ and, assuming that $j_l$ has been constructed for every $l < n$, we choose $j_{n}$ large enough so that
		\begin{equation} \label{cond1}
		\sum_{l<n} 2^{(1-\alpha) j_l} < \frac{1}{2} 2^{(1-\alpha) j_n} 
		\end{equation}
		but also
		\begin{equation} \label{cond3}
		\alpha j_n \geq \alpha j_{n-1} + 5  
		\end{equation}
		and
		\begin{equation} \label{cond2}
		\alpha r_{n-1} j_n >\big(1 - r_{n-1} (1- \alpha ) \big)  j_{n-1} +n \, .
		\end{equation}
		Then, we set
		$$
		c_{j,k} = \left\{
		\begin{array}{ll}
			2^{-\alpha j} & \mbox{if there exists $n \in \bN_0$ such that } j = j_n, \\[1ex]
			0 & \mbox{otherwise}, 
		\end{array}
		\right. 
		$$
		and we consider the function $F$ defined by
		$$
		F(x) = \sum_{j \in \bN_0} \sum_{k=0}^{2^j -1} c_{j,k} \Lambda (2^j x-k) \, .
		$$
		From the definition of the coefficients $c_{j,k}$ and the localization of the support of $\Lambda$, it is clear that this series is uniformly convergent on $[0,1]$, hence $F$ is well defined. Let us first show, using standard arguments, that $F$ belongs to $C^\alpha [0,1]$. Let us fix $x,y \in [0,1]$ distinct and let us consider $J \in \mathbb N_0$ such that 
		\begin{equation}\label{eq:choiceJ}
		2^{-J-1}< |x-y| \leq 2^{-J}.
		\end{equation}
		We have
		\begin{align} \label{eq:holder}
			\left| F(x) - F(y)\right| \leq  & \left| \sum_{j < J} \sum_{k=0}^{2^j -1} c_{j,k} \big(\Lambda (2^j x-k) - \Lambda(2^j y -k) \big )\right| \\ 
			& +  \left| \sum_{j \geq J}  \sum_{k=0}^{2^j -1} c_{j,k} \Lambda (2^j x-k) \right| + \left| \sum_{j \geq J} \sum_{k=0}^{2^j -1} c_{j,k} \Lambda (2^j y-k) \right|.\nonumber
		\end{align}
		Let us estimate the first term of (\ref{eq:holder}). 
		Since $\Lambda$ is $1$-Lipschitz, for any $j\in\mathbb{N}_0$ and $k\in \{ 0,\ldots,2^j-1 \}$, we have
	\[
	|\Lambda(2^j x - k) - \Lambda(2^j y - k)| \le 2^j |x-y|.
	\]
	Moreover, for each fixed $j\in\mathbb{N}_0$, there are at most two values of $k\in \{ 0,\ldots,2^j-1 \}$ for which $\Lambda(2^j x - k)$ or $\Lambda(2^j y - k)$ is nonzero, since the support of $\Lambda(2^j \cdot - k)$ has length $2^{-j}$. Therefore, the first term in \eqref{eq:holder} can be bounded as
	\begin{align}\label{eq:holder1}
	\left| \sum_{j < J} \sum_{k=0}^{2^j -1} c_{j,k} (\Lambda(2^j x - k) - \Lambda(2^j y - k)) \right|
	& \le \sum_{j < J} 2 \, 2^{-\alpha j} 2^j |x-y| \nonumber \\
	&  \le C 2^{(1-\alpha)J} |x-y| \nonumber\\
	& <  2^{\alpha }C |x-y|^{\alpha}
	\end{align}
	for $C=\frac{2}{2^{1-\alpha}-1} >0$ is a constant independent of $x$ and $y$, and where we have used \eqref{eq:choiceJ}.
	Indeed, firstly
		$$
		\sum_{j<J} 2^{(1-\alpha)j} = \frac{2^{(1-\alpha)J}-1}{2^{(1-\alpha)}-1} \leq \frac{2^{(1-\alpha)J}}{2^{1-\alpha}-1} 
		$$
	and secondly, by \eqref{eq:choiceJ},
		$$
		2^{(1-\alpha)J} |x-y| = 2^{J} 2^{-\alpha J} |x-y| < \frac{1}{|x-y|} 2^\alpha |x-y|^\alpha |x-y| = 2^\alpha |x-y|^\alpha.
		$$
	 For the second term of (\ref{eq:holder}), using again \eqref{eq:choiceJ}, we obtain
		\begin{align}\label{eq:holder2}
		\left| \sum_{j \geq J} \sum_{k=0}^{2^j -1} c_{j,k} \Lambda (2^j x-k) \right| & \leq \sum_{j \geq J} 2^{-\alpha j} = \frac{2^{-\alpha J}}{1-2^{-\alpha}} \\
		& \leq \frac{2^\alpha}{1-2^{-\alpha}} |x-y|^\alpha = C' |x-y|^{\alpha}
		\end{align}
		for $C'=\frac{2^\alpha}{1-2^{-\alpha}}>0$ a constant independent of $x$ and $y$. 
		Analogously, the third term of (\ref{eq:holder}) can be estimated by
		\begin{equation}\label{eq:holder3}
			\left| \sum_{j \geq J} \sum_{k=0}^{2^j -1} c_{j,k} \Lambda (2^j y-k) \right|  \leq C' |x-y|^{\alpha}.
		\end{equation}
		Putting together (\ref{eq:holder}), (\ref{eq:holder1}), (\ref{eq:holder2}) and (\ref{eq:holder3}), we obtain that $F \in C^{\alpha} [0,1]$. 
		
		\medskip
		Let us now prove that the function $F$ satisfies the second assumption of Proposition~\ref{prop_dimgraph}. Let us fix $0<\delta <1$ and $x \in [0,1)$ (with obvious adaptation if $x=1$). Then, there is $n \in \mathbb N_0$ such that
		\begin{equation}\label{eq:delta}
			2^{-j_n-1} \leq \delta < 2^{-j_n}
		\end{equation}
		and $k \in \{0, \dots, 2^{j_n}-1\}$ such that $x \in \left[ \frac{k}{2^{j_n}}, \frac{k+1}{2^{j_n}}\right)$. Consider now $y \in \left[ \frac{k}{2^{j_n}}, \frac{k+1}{2^{j_n}}\right)$ such that
		\begin{equation}\label{eq:dist}
			|x-y| \geq 2^{-j_n-1}. 
		\end{equation}
		Since $x$ and $y$ belongs to the same dyadic interval of length $2^{-j_n}$, we have
		\begin{eqnarray}\label{eq:exp1}
			\left|  \sum_{j < j_n}\sum_{k=0}^{2^j -1} c_{j,k} \big(\Lambda (2^j x-k) - \Lambda(2^j y -k) \big )\right| & \leq & 
			\sum_{l<n} 2^{(1-\alpha) j_l}   |x-y| \nonumber \\
			& < & \frac{1}{2} 2^{(1-\alpha) j_n} |x-y|
		\end{eqnarray}
		by using condition (\ref{cond1}). Moreover, again since $x$ and $y$ belong to the same dyadic interval $\left[ \frac{k}{2^{j_n}}, \frac{k+1}{2^{j_n}}\right)$ (which in particular it means that $x$ and $y$ are either in $[0,1/2)$ or in $[1/2,1)$), we have
		\begin{align}\label{eq:exp2}
			\left|  \sum_{k=0}^{2^{j_n} -1} c_{j,k} \big(\Lambda (2^{j_n} x-k) - \Lambda(2^{j_n} y -k) \big )\right| & = \left|  c_{j,k} \big(\Lambda (2^{j_n} x-k) - \Lambda(2^{j_n} y -k) \big )\right| \nonumber \\
			& = c_{j,k} 2^{j_n} |x-y| \nonumber \\
			& = 2^{(1-\alpha) j_n} |x-y|.
		\end{align}
		Thus, from \eqref{eq:exp1} and \eqref{eq:exp2}, we obtain
		\begin{equation}\label{eq:exp3new}
		\left|  \sum_{j \leq j_n}\sum_{k=0}^{2^j -1} c_{j,k} \big(\Lambda (2^j x-k) - \Lambda(2^j y -k) \big )\right| > \frac{1}{2} 2^{(1-\alpha) j_n} |x-y| \geq \frac{1}{2} 2^{-\alpha j_n} |x-y|.
		\end{equation}
		Now, by applying induction on \eqref{cond3}, we have that $\alpha j_{l+k} \geq \alpha j_l+5k$ for any $l\in \bN_0$ and $k\in \bN$, which yields
			$$
			2^{-\alpha j_{l+k}} \leq 2^{-\alpha j_l} \cdot 2^{-5k}.
			$$
		Hence,
		\begin{align}\label{eq:exp3}
			\left|  \sum_{j > j_n}\sum_{k=0}^{2^j -1} c_{j,k} \big(\Lambda (2^j x-k) - \Lambda(2^j y -k) \big )\right| & \leq 2 \cdot \sum_{l>n} 2^{-\alpha j_l} \leq 2\cdot 2^{-\alpha j_{n+1}} \cdot \sum_{k\in \bN_0} 2^{-5k} \nonumber \\
			& = 2\cdot 2^{-\alpha j_{n+1}} \cdot \frac{1}{1-2^{-5}} = 2\cdot \frac{32}{31} \cdot 2^{-\alpha j_{n+1}} \nonumber \\
			& < 4 \cdot 2^{-\alpha j_{n+1}}.
		\end{align}
		We obtain from \eqref{eq:exp3new} and \eqref{eq:exp3} that
		\begin{equation}\label{eq:difference}
			\big| F(x) - F(y) \big|  > \frac{1}{2} 2^{(1-\alpha) j_n} |x-y| - 4 \cdot 2^{-\alpha j_{n+1}} .
		\end{equation}
		By using \eqref{eq:delta} and \eqref{eq:dist} with the condition \eqref{cond3}, we obtain then
		$$
		\big| F(x) - F(y) \big|  > \frac{1}{2}2^{- \alpha j_{n} } - 4 \cdot 2^{- \alpha j_{n+1}} > \frac{1}{2}\delta^{\alpha} - 4 \cdot 2^{- \alpha j_n} \cdot 2^{-5} > \frac{1}{2}\delta^{\alpha} - \frac{1}{8} \delta^{\alpha} = \frac{3}{8}\delta^{\alpha}.
		$$
		Hence, $F$ is strongly monoH\"older on $[0,1]$. 
		
		\medskip
		
		In order to conclude, we still have to prove that the level sets of $F$ have Hausdorff dimension zero. Let $E$ denote a set on which $F$ is constant, and let us fix $r >0$. For every $n \in \bN$, we have
		$$
		E = \bigcup_{k=0}^{2^{j_n+1}-1} E_{n,k}  \quad \mbox{where} \quad E_{n,k} = \left[ \frac{k}{2^{j_n+1}}, \frac{k+1}{2^{j_n+1}}\right) \cap E.
		$$
		Since $\diam(E_{n,k}) \leq  \frac{1}{2^{j_n+1}} $, if we show that
		$$
		\sum_{k=0}^{2^{j_n+1}-1} \diam(E_{n,k}) ^r  \longrightarrow 0 \quad \mbox{as} \quad  n \to \infty ,
		$$
		then we will get that $\mathcal{H}^r (E) =0$. Let us consider $x,y \in E_{n,k}$.  From \eqref{eq:difference}, we already know that 
		$$
		0 = \big| F(x) - F(y) \big|  \geq \frac{1}{2} 2^{(1-\alpha) j_n} |x-y| - 4 \cdot 2^{-\alpha j_{n+1}} ,
		$$
		hence
		$$
		|x-y| \leq 8 \cdot  2^{-\alpha j_{n+1} +(\alpha- 1) j_n}.
		$$
		We obtain then that
		$$
		\sum_{k=0}^{2^{j_n+1}-1} \diam(E_{n,k}) ^r  \leq 8^r \cdot 2 \cdot 2^{-r \alpha j_{n+1} +\left(1-r (1- \alpha) \right)j_n} .
		$$
		The sequence $( 2^{-r \alpha j_{n+1} +(1  - r (1- \alpha)) j_n})_{n \in \bN}$ converges to $0$ as $n$ tends to infinity.
		Indeed, for $n$ large enough, we have $r_n \leq r$ and using condition (\ref{cond2}), we obtain
		$$
		r \alpha j_{n+1} \geq \alpha r_{n} j_{n+1} >\big(1-r_{n} (1- \alpha )\big)  j_{n} +n+1  > \big(1-r (1- \alpha )\big) j_{n} +n +1 \, .
		$$
		The conclusion follows.
	\end{proof}

	\begin{remark}In particular, $h_F(x_0) = \alpha$ for every
		$x_0 \in [0,1]$ and  ${\dim}_B (F) =  2-
		\alpha$. Indeed, it is directly deduced from  the characterization
		of the H\"older regularity in the Schauder basis and Proposition \ref{prop_dimgraph}.
	\end{remark}

	Let us  now present the main result of this section.

	\begin{theorem}\label{thm:alg}
		Let $\alpha \in  (0,1)$. The set of functions $f \in C^{\alpha}[0,1]$ with $\dim_{B}(f) = 2- \alpha$  and for which there exists $E \subseteq [0,1]$ such that $\dim_{\mathcal{H}}(E) = 0$ and $h_f(x_0)= \alpha$ for every $x_0 \in [0,1] \setminus E$ is strongly $\mathfrak{c}$-algebrable.
	\end{theorem}

	\begin{proof}[Proof of Theorem \ref{thm:alg}]
		We will use the technique described in Proposition \ref{prop_explike}. Let us consider the function $F$ constructed in Lemma \ref{lemma_constr} and any exponential-like function $g$. Let us also fix $x_0 \in [0,1]$. For every $x \in [0,1]$, we have
		\begin{equation}\label{eq:holdercompo}
			g \circ F(x) - g \circ F (x_0) = g'(t) \big( F(x) -F(x_0) \big)
		\end{equation}
		for some $t$ between $F(x)$ and $F(x_0)$. In particular, $g \circ F
		\in C^\alpha [0,1]$. Moreover, if $g'(F(x_0)) \neq 0$, one directly gets that
		$$
		h_{g \circ F}(x_0) = h_F(x_0) = \alpha.
		$$
		Let us note that the derivative $g'$ of $g$ is also an exponential-like function. Therefore, using Lemma  \ref{lemma_explike}, we know that its preimage $(g')^{-1}(\{0\})$ has finitely many elements. From the construction of $F$, we get that the set $E$ of points $x_0$ for which $g'(F(x_0)) = 0$ has Hausdorff dimension $0$.

		For the computation of the graph dimension, we already know that 
		$$
		\overline{\dim}_B (g \circ F) \leq 2- \alpha 
		$$
		using the first  part of Proposition \ref{prop_dimgraph}. Moreover,
		the continuity of $g'$ and $F$ together with the definition of $E$
		give the existence of an interval $I \subseteq [0,1] \setminus E$ such that
		$$
		\inf_{u \in I} |g'(F(u))| > 0.
		$$
		Then, since $F$ is strongly monoH\"older on $I$ with exponent
		$\alpha$, \eqref{eq:holdercompo} implies that the function $g \circ F$ is also strongly monoH\"older on $I$ with exponent $\alpha$. In order to conclude, it suffices to remark that 
		$$
		\underline{\dim}_B (g \circ F) \geq \underline{\dim}_B (g \circ F \vert_{I}).
		$$
		and to use the second part of Proposition \ref{prop_dimgraph}. 
	\end{proof}

	As a consequence of the proof of Theorem~\ref{thm:alg} and the relation established between the property of strongly monoH\"older and nowhere differentiability in Subsection~\ref{fracder}, we obtain the following result.

	\begin{theorem}\label{thm:fracalg}
		Let $\alpha \in (0,1)$.
		There exists an open interval $I \subseteq [0,1]$ such that the set of functions that are nowhere fractional differentiable on $I$ for any order in $(\alpha,1)$ and fractional differentiable on $I$ for any order in $(0,\alpha)$ is strongly $\mathfrak c$-algebrable.
	\end{theorem}

%%%%%%%%%%%%%%%%%%%%%%%%%%%%%%%%%%%%%%%%%%%%	

\end{document}